\documentclass[article]{article}
\usepackage[english]{babel}
\usepackage[utf8]{inputenc}
\usepackage[T1]{fontenc}
\usepackage{amsmath}
\usepackage{amsthm}
\usepackage{mathtools}
\usepackage{amssymb}
\usepackage{tikz-cd}
\usepackage{amsfonts}
\usepackage{thmtools,xcolor}
\usepackage{hyperref}
\usepackage{cleveref}
\usepackage{enumitem}

\allowdisplaybreaks[1]

\usepackage[a4paper]{geometry}
\theoremstyle{plain}
\newtheorem{theorem}{Theorem}[section]
\newtheorem{cor}[theorem]{Corollary}

\newtheorem{prop}[theorem]{Proposition}

\theoremstyle{definition}
\newtheorem{examp}[theorem]{Example}
\newtheorem{examps}[theorem]{Examples}
\newtheorem{dfn}[theorem]{Definition}
\newtheorem{rem}[theorem]{Remark}
\newtheorem{rems}[theorem]{Remarks}

\theoremstyle{remark}

\makeatletter
\providecommand*{\twoheadrightarrowfill@}{%
  \arrowfill@\relbar\relbar\twoheadrightarrow
}
\providecommand*{\twoheadleftarrowfill@}{%
  \arrowfill@\twoheadleftarrow\relbar\relbar
}
\providecommand*{\xtwoheadrightarrow}[2][]{%
  \ext@arrow 0579\twoheadrightarrowfill@{#1}{#2}%
}
\providecommand*{\xtwoheadleftarrow}[2][]{%
  \ext@arrow 5097\twoheadleftarrowfill@{#1}{#2}%
}
\makeatother

\makeatletter
\def\ind{\@ifnextchar[{\@with}{\@without}}
\def\@with[#1]#2{\mathrm{Ind}(#1,#2)}
\def\@without#1{\mathrm{Ind}(#1)}
\makeatother

\newcounter{para}[section]

\newcommand\KK[0]{K\! K}

\DeclareMathOperator\dnc{DNC}

\newcommand{\R}{\mathbb{R}}

\begin{document}
\title{On the deformation groupoid of the inhomogeneous pseudo-differential Calculus}
\author{Omar Mohsen}
\date{}
\maketitle
\abstract{Recently, Van-Erp and Yuncken and independently Choi and Ponge defined an inhomogeneous deformation groupoid. As shown by Van-Erp and Yuncken, this deformation groupoid allows to fully recovers the general inhomogeneous pseudo-differential calculus.

In this article we simplify and generalise this construction using a double (multiple) deformation to the normal cone.}



\section*{Introduction} In order to construct a parametrix for Hörmander's \cite{MR0222474} subelliptic operators on a contact manifold, Folland and Stein \cite{MR0344699,MR0367477} defined a noncommutative pseudo-differential calculus where the principal cosymbol is a function on a bundle of Heisenberg groups. A fundamental characteristic of this pseudo-differential calculus is that a vector field defines a differential operator of order $1$ if it is everywhere tangent to the contact subbundle and of order $2$ if not. Later on, this was generalised to an arbitrary subbundle of the tangent bundle, and even further to a filtration of the tangent bundle under conditions on the Lie bracket (see \cite{MR0493005,MR0370271,MR660648,MR973272,MR1185817,MR764508,MR0423432,MR953082,MR0442149,MR0436223}). To such a structure one associates a bundle of graded nilpotent Lie groups over which the cosymbols are functions. Let us remark that the general situation is more involved because the bundle of graded nilpotent Lie groups doesn't need to be locally trivial and the analogue of the theorem of Darboux doesn't hold in general.

This calculus was later used by many authors in index theory and $C^*$-algebras, for instance by Connes and  Moscovici \cite{MR1657389,MR1334867} to define a transversal signature operator on foliated manifolds and do computations in cyclic cohomology, following a construction of Hilsum and Skandalis \cite{MR925720}, and by Julg and Kasparov \cite{MR1332908} to compute the $SU(n,1)$ equivariant $\KK$-theory following the work of Rumin  \cite{MR1267892}.

In \cite{MR3187645}, Debord and Skandalis gave a global definition of the classical pseudo-differential calculus thanks to the tangent groupoid of Connes. This definition was adapted to the case of inhomogeneous calculus by Van-Erp and Yuncken  \cite{Erp:2015aa}. To this end they used an inhomogeneous deformation groupoid\footnote{Following Ponge's recommendation, the inhomogeneous deformation groupoid will be called the Carnot groupoid.} instead of Connes's tangent groupoid. This inhomogeneous groupoid was constructed in the case of contact manifolds by Ponge and van-Erp independentaly \cite{MR2247877,MR2680395}, and in the general case of filtrations by Choi and Ponge \cite{Choi:2015aa,Choi:2017aa,Ponge:2017aa} and van Erp and Yuncken \cite{Erp:2016aa} following work by Julg and van Erp \cite{MR3694098}.

 This groupoid was also used by van Erp \cite{MR2680396,MR2680395} and later (with Baum \cite{MR3261009}) to formulate and prove an index formula in the same spirit as that of Atiyah-Singer. Their index theorem is for differential operators whose cosymbol is invertible in the above calculus associated to a contact structure. These operators are necessarily hypoelliptic, hence their analytic index is well defined but they are rarely elliptic. It was also used by van Erp \cite{MR2746652} to formulate and prove an index theorem for hypoelliptic operators on foliated manifolds.

In the present article, we give an elementary construction of this Carnot groupoid using the deformation to the normal cone construction. Our approach gives rise to noncommutative Lie groupoids/symbols precisely because we deform Lie groupoids with respect to subgroupoids and not with respect to spaces, and contrary to the methods used in \cite{MR2247877,MR2680395,MR2417549,Choi:2015aa,Choi:2017aa,Ponge:2017aa,Erp:2016aa} no analysis on local coordinates is needed to construct the Lie groupoid, only functoriality of the $\dnc$ construction, furthermore no analysis on higher jets is needed to construct the Lie groupoid.

\smallskip

We now briefly describe our construction.

 Let us first recall the deformation to the normal cone construction. If $V\subseteq M$ is submanifold, then the set $\dnc(M,V)=M\times\R^*\sqcup \mathcal{N}^M_V\times\{0\}$ admits a natural smooth structure, where $\mathcal{N}^M_V$ is the normal bundle. It then follows from the functoriality of the construction that if $H\rightrightarrows H^0$ is a Lie subgroupoid of $G\rightrightarrows G^0$, then $\dnc(G,H)$ is naturally a Lie groupoid over $\dnc(G^0,H^0)$. Connes's tangent groupoid is then the groupoid $\dnc(M\times M,M)\rightrightarrows M\times \R$.

Our construction of the inhomogenuous deformation groupoid is as follows. Let $H\subseteq TM$ be a vector bundle. Recall the tangent groupoid $$\dnc(M\times M,M)=M\times M\times \R^*\sqcup TM\times \{0\}\rightrightarrows M\times \R$$ defined by Connes. The space $H\times \{0\}\subseteq TM\times \{0\}\subseteq\dnc(M\times M,M)$ is a Lie subgroupoid. Hence by the Functoriality of the $\dnc$ construction, the space $$\dnc(\dnc(M\times M,M),H\times\{0\})\rightrightarrows \dnc(M\times \R,M\times \{0\})=M\times \R^2$$ is a Lie groupoid. We prove that the fiber over $M\times \{1\}\times \R$ is the Canot Lie groupoid. Furthermore the groupoid  $\dnc(\dnc(M\times M,M),H\times\{0\})\rightrightarrows M\times \R^2$ is a quite natural object to study because it contains `the deformations in all the directions'. In the case of a $2$-step filtration $H^1\subseteq H^2\subseteq TM$ with the hypothesis $[H^1,H^1]\subseteq H^2$, we construct a Lie groupoid as follows: by the previous construction with $H=H^1$, we have a Lie groupoid $$M\times M\times \R^*\sqcup H^1\oplus TM/H^1\times\{0\},$$ with a nilpotent group structure on $H^1\oplus TM/H^1$. The condition $[H^1,H^1]\subseteq H^2$ is then precisely the condition needed for $$H^1\oplus H^2/H^1\times\{0\}\subseteq M\times M\times \R^*\sqcup H^1\oplus TM/H^1\times\{0\}$$ to be a Lie subgroupoid. Hence by functoriality of $\dnc$ the space \begin{align*}
 \dnc( M\times M\times \R^*\sqcup H^1\oplus TM/H^1\times\{0\},H^1\oplus H^2/H^1\times\{0\})\rightrightarrows M\times \R^2
\end{align*} is a Lie groupoid. We restrict to $M\times \{1\}\times \R$ to obtain the deformation groupoid associated to the filtration. The general case is then treated in \cref{Heisenberg groupoid sect} by induction.

The methods developed here can be used to give a variety of examples of Lie groupoids which can be used to define an associated inhomogenuous pseudo-differential calculi in a variety of geometric situations. In particular we extend the Carnot Lie groupoid to cover the case of transverse (to a foliation) hypoelliptic pseudo-differential calculus without any difficulty (examples \ref{example H foliation single} and \ref{examp foliation multipleHeisen}).

\smallskip

This article is organised as follows.  

In \Cref{Premlim sect}, some preliminaries are recalled.

In \Cref{dnc section}, we recall the notion of the deformation to the normal cone following \cite{PCaro, Debord:2017aa}.

 In \Cref{Dnc iterated sect}, the iterated deformation to the normal cone construction is introduced.
 
In \Cref{prem sect prop}, a proposition is proved which will be used in \Cref{M single H examp}, in order to give us the algebraic structure of the symbol part of the Carnot groupoid.
 
\smallskip 

In \Cref{sect case 1} the case of a single bundle is treated thoroughly.

In \Cref{M single H examp}, we define the Carnot groupoid and we calculate the Lie groupoid structure of the symbol part of Carnot groupoid.

  In \Cref{sect H bundle 1 case}, as the construction in \cite{Choi:2015aa,Ponge:2017aa,Choi:2017aa,Erp:2016aa} of are based on local charts, we describe our construction locally, and show that the two construction agree.
 
  \smallskip

  In \Cref{Heisenberg groupoid sect}, we generalize the construction given in \Cref{M single H examp} but for a filtration of the tangent bundle proving that iterated deformation to the normal cone gives rise to the Carnot groupoid in the general case. This section is independent of \Cref{M single H examp} and provides another proof of \Cref{main thm H single}.
  
  \smallskip
  
  The paper ends with a paragraph on a related construction of Sadegh and Higson \cite{Sadegh:2016aa} seen here as the quotient of a Lie groupoid by a Lie subgroupoid.
\section*{Acknowledgments}
I wish to express my gratitude to my thesis director, G. Skandalis, for his support and his numerous insights and remarks into this work. I would also like to thank Prof. Higson and Prof. Ponge for their kind remarks and recommendations. This work was supported by grants from Région Ile-de-France.
\section{Preliminaries}\label{Premlim sect}
\subsection{Deformation to the normal cone construction}\label{dnc section}
In this section, we recall the deformation to the normal cone construction following \cite{PCaro, Debord:2017aa}. \textit{The deformation to the normal cone}\index{deformation to the normal cone} of a manifold $M$ along an immersed submanifold $V$ is a manifold whose underlying set is $$\dnc(M,V):=M\times \R^*\sqcup N^M_V\times\{0\},$$ where $N^M_V$ is the normal bundle of $V$ inside $M$. The smooth structure is defined by covering $\dnc(M,V)$ with two sets. The first is $M\times \R^*$. The second is $\phi(N^M_V)\times\R^*\sqcup N^M_V\times\{0\}$ where $\phi:N^M_V\to M$ is a tubular embedding.\footnote{To simplify the exposition, we will always assume that tubular neighbourhoods are diffeomorphisms on $N^M_V$. In the case of immersed manifolds, the tubular neighbourhoods are only local in $V$.} The smooth structure on $\phi(N^M_V)\times\R^*\sqcup N^M_V\times\{0\}$ is given by declaring the following bijection a diffeomorphism \begin{align*}
&\tilde{\phi}:N^M_V\times \R\to \phi(N^M_V)\times\R^*\sqcup N^M_V\times\{0\}\\
 &\tilde{\phi}(x,X,t)=(\phi(x,tX),t)\in M\times \R^*,\quad t\neq 0\\
 &\tilde{\phi}(x,X,0)=(x,X,0)\in N^M_V\times\{0\}.
\end{align*}
This smooth structure is independent of $\phi$. This follows by noticing that the following functions are smooth functions that generate the smooth structure:\begin{enumerate}
\item the function\begin{align*}
(\pi_M,\pi_\R):\dnc(M,V)&\to M\times \R\\
(x,t)&\to (x,t),\quad t\neq 0\\
(x,X,0)&\to (x,0)
\end{align*}
\item Let $f\in C^\infty(M)$ be a smooth function which vanishes on $V$. Therefore $df:N^M_V\to \R$ is well defined. The following function is smooth
\begin{align*}
\dnc(f):\dnc(M,V)&\to \R\\
(x,t)&\to \frac{f(x)}{t},\quad t\neq 0\\
(x,X,0)&\to df_x(X)
\end{align*}
\end{enumerate}

The group $\R^*$ acts smoothly on $\dnc(M,V)$. The action is given by $\lambda_{u}(x,t)=(x,ut)$ and $\lambda_u(x,X,0)=(x,\frac{X}{u},0)$ for $u\in \R^*$. 

\bigskip

\begin{prop}[Functoriality of DNC]\label{Functoriality of DNC}
Let  $M,M'$ be smooth manifolds, $V\subseteq M$, $V'\subseteq M'$ submanifolds, $f:M\to M'$ a smooth map such that $f(V)\subseteq V'$. Then the map defined by 
\begin{align*}
\dnc(M,V)&\to \dnc(M',V')\\
(x,t)&\to (f(x),t),\quad t\neq 0\\
(x,X,0)&\to (f(x),df_x(X),0)
\end{align*}
is a smooth map\footnote{In the case where $V'$ is an immersed submanifold, one must also suppose that $f_{|V}:V\to V'$ is continuous.} that will be denoted by $\dnc(f)$. 
 Furthermore the map  $\dnc(f)$ is \begin{itemize}
\item  a submersion if and only if $f$ is a submersion and $f_{|V}:V\to V'$ is also a submersion.
 \item an immersion if and only if $f$ is an immersion and for every $v\in V$, $T_vV=df_v^{-1}(TV').$
 \end{itemize}
\end{prop}
\begin{proof}
Smoothness of $\dnc(f)$ follows from the description of smooth maps given above. For statements concerning submersions and immersions. Let $U\subseteq \dnc(M,V)$ be the set where the differential of $\dnc(f)$ is onto (respectively injective). It is clear that $U$ is an open set that is invariant under the $\R^*$ action and contains $M\times\R^*$. To prove that $U=\dnc(M,V)$, it suffices to prove that $V\times \{0\}\subseteq U.$ If $v\in V$, then one sees directly that $$T_{(v,0)}\dnc(M,V)=\R\oplus T_vV\oplus T_vM/T_vV$$ The differential of $\dnc(f)$ is then $d\dnc(f)_{(v,0)}(t,X,Y)=(t,df_v(X),df_v(Y)).$ The proposition is then clear.
\end{proof}
The map \begin{align*}
 N^M_V\to N^{M'}_{V'},\quad (x,X)&\to (f(x),df_x(X)) 
\end{align*} will be denoted by $Nf$.

\begin{rem}\label{rem nonhaus} More generally if $V$ is a smooth manifold, $i:V\to M$ is an immersion but not necessarily injective, then the manifold $\dnc(M,V)$ can still be defined using the same charts as above. The main difference is that the manifold $\dnc(M,V)$ is no longer Hausdorff.
\end{rem}
\begin{rem}\label{action dnc}It follows from \Cref{Functoriality of DNC}, that if $G$ is a Lie group acting smoothly on a manifold $M$ that leaves a submanifold $V$ invariant, then $G$ acts smoothly on $\dnc(M,V)$. This action commutes with the $\R^*$ action $\lambda$. In particular the group $G\times \R^*$ acts on $\dnc(M,V).$
\end{rem}
\begin{prop}\label{DNC fibered product}
Let $M_1, M_2, M$ be manifolds, $V_i\subseteq M_i,V\subseteq M$ submanifolds, $f_i:M_i\to M$ smooth maps such that \begin{enumerate}
\item $f_i(V_i)\subseteq V$ for $i\in \{1,2\}$
\item the maps $f_i$ are transverse
\item the maps $f_{i}|V_i:V_i\to V$ are transverse
\end{enumerate} Then \begin{enumerate}
\item\begin{enumerate}[label=(\alph*)]
\item the maps $Nf_i:N^{M_i}_{V_i}\to N^M_V$ are transverse.
\item the natural map $$N^{M_1\times_MM_2}_{V_1\times_VV_2}\to N^{M_1}_{V_1}\times_{N^{M}_{V}}N^{M_1}_{V_2}$$ is a diffeomorphism.

\end{enumerate} 
Similarly for $\dnc$, we have \item\begin{enumerate}[label=(\alph*)]
 \item the maps $\dnc(f_i):\dnc(M_i,V_i)\to \dnc(M,V)$ are transverse.
 \item the natural map $$\dnc(M_1\times_MM_2,V_1\times_VV_2)\to \dnc(M_1,V_1)\times_{\dnc(M,V)}\dnc(M_2,V_2)$$ is a diffeomorphism.
\end{enumerate}

\end{enumerate}
\end{prop}
\begin{proof}
Let us prove $1.$ $(a)$. The group $\R^*$ actson $N^M_{V_i}$ and $N^M_V$. Both $Nf_i$ are $\R^*$ equivariant. Since transversality is an open condition it follows that it suffices to check transversaility at the origin. In that case for $x_1\in V_1$, $x_2\in V_2$ such that $f_1(x_1)=f_2(x_2)$, one has $$T_{(x_i,0)}N^M_{V_i}=T_{x_i}V_i\oplus N^M_{V_i,x_i}.$$ Transversality of $f_i$ at $(x_i,0)$ follows from condition $2$ and $3$.

Statement $1.$ $(b)$ and bijectivity of the natural map $$\dnc(M_1\times_MM_2,V_1\times_VV_2)
 \to \dnc(M_1,V_1)\times_{\dnc(M,V)}\dnc(M_2,V_2)$$ directly follow from statement $1.$ $(a)$. To prove that it is a diffeomorphism and that the maps $\dnc(f_i)$ are transverse, we use the same argument as in \Cref{Functoriality of DNC}. The two conditions are open conditions which are $\R^*$-invariant. Hence it suffices to check that them at $V_1\times_V V_2$ which follows directly from $1.$ $(a)$ and $1.$ $(b)$.
\end{proof}
\begin{prop}\label{last min prop}Let $V\subseteq M$ a submanifold, then \begin{enumerate}
 \item $TN^M_V=N^{TM}_{TV}$ 
 \item if $\pi_\R:\dnc(M,V)\to \R$ is the natural projection, then $\ker(d\pi_\R)=\dnc(TM,TV).$
\end{enumerate}
\end{prop}
\begin{proof}
We will define a diffeomorphism $\phi:TN^M_V\to N^{TM}_{TV}$. Let $c(t,s):\R^2\to M$ be a smooth map such that $c(t,0)\in V$ for all $t$. It follows that for each fixed $t$, $\partial_s c(t,0)\in N^M_V$, and hence $\partial_t\partial_s c(0,0)\in TN^M_V.$ Conversely for each fixed $s$, $\partial_t c(0,s)\in TM$ and its value at $s=0$ is in $TV$, hence $\partial_s\partial_t c(0,0)\in N^{TM}_{TV}$. The map $\phi$ is the map sending $\partial_t\partial_s c(0,0)$ to $\partial_s\partial_t c(0,0)$, for each path $c$. One easily checks in local coordinates this $\phi$ is indeed a diffeomorphism.

For the deformation to the normal cone, by definition $\ker(d\pi_\R)=TM\times \R^*\sqcup TN^M_V.$ Using the map $\phi$, one defines a map $\ker(d\pi_\R)\to \dnc(TM,TV).$ It is straightforward to check that the resulting map is a difffeomorphism by checking so in local coordinates.
\end{proof}
Let us recall the notion of a $\mathcal{VB}$-groupoid from \cite{MR941624,MR896907}.
\begin{dfn}\label{dfn of VB groupoids}Let $H$ be a Lie groupoid. A $\mathcal{VB}$-groupoid over $H$ is given by \begin{itemize}
\item a vector bundle $G$ over $H$
\item a vector bundle $G^0$ over $H^0$
\item a Lie groupoid structure on $G$ whose space of objects is $G^0$, such that the map range map $r:G\to G^0$, the inverse map $i:G\to G$, the multiplication map $m:G\times_{s,r}G\to G$ are respectively bundle maps over the range map $r:H\to H^0$, the inverse map $i:H\to H$ and the multiplication map $m:H\times_{s,r}H\to H$.
\end{itemize}
By abuse of notation we will call $G\rightrightarrows G^0$ a $\mathcal{VB}$-groupoid over $H$.

A $\mathcal{VB}$-subgroupoid of $G$ is a Lie subgroupoid $K\rightrightarrows K^0$ of $G$ such that $K$ is a subbundle of $G$ and $K^0$ is a subbundle of $G^0$. 
\end{dfn}
\begin{theorem}\label{Construction of dnc grouopids}Let $G$ be a Lie groupoid, $H$ a Lie subgroupoid. Then
\begin{enumerate}
\item the space $N^G_H\rightrightarrows N^{G^0}_{H^0}$ is a Lie groupoid whose structure maps are $Ns$, $Nr$ and whose Lie algebroid is equal to $N^{\mathfrak{A}G}_{\mathfrak{A}H}.$ Furthermore, $N^G_H$ is a $\mathcal{VB}$-groupoid over $H$.
\item  the manifold $\dnc(G,H)\rightrightarrows\dnc(G^0,H^0)$ is a Lie groupoid whose structure maps are $\dnc(s)$, $\dnc(r)$ and Lie algebroid is equal to $\dnc(\mathfrak{A}G,\mathfrak{A}H)$.
\item if $K\subseteq H$ is a Lie subgroupoid, then the restriction of the normal bundle $N^G_H\big|_{K}\rightrightarrows N^{G^0}_{H^0}\big|_{K^0}$ is a Lie subgroupoid of $N^G_H\rightrightarrows N^{G^0}_{H^0}$ whose Lie algebroid is $N^{\mathfrak{A}G}_{\mathfrak{A}H}\big|_{\mathfrak{A}K}.$ Furthermore $N^G_{H}\big|_{K}$ is a $\mathcal{VB}$-groupoid over $K$.
\end{enumerate}
\end{theorem}
\begin{proof}
The fact that $N^G_H\rightrightarrows N^{G^0}_{H^0}$ and $\dnc(G,H)$ are Lie groupoids is a direct consequence of propositions \ref{Functoriality of DNC} and  \ref{DNC fibered product}. For example the product map of $\mathcal{N}^G_H\rightrightarrows N^{G^0}_{H^0}$ is defined using proposition \ref{DNC fibered product}. If $M:G^{(2)}\to G$ denotes the product map then $$NM:N^{G^{(2)}}_{H^{(2)}}\to N^G_H$$ is well defined. Now using proposition \ref{DNC fibered product}, it follows that $$N^{G^{(2)}}_{H^{(2)}}=N^G_H\times _{ds,dr}{N^G_H}.$$ Hence $NM$ can be identified with a map $$N^G_H\times _{ds,dr}{N^G_H}\to N^G_H$$. This is by definition the product map of $\mathcal{N}^G_H$. Axioms like associativity and the identity  all follow from the fact that the corresponding axioms hold for $G$ and the functoriality of the constructions $N$ and $\dnc$.

 For the Lie algebroid computations, since $\mathcal{N}^G_H$ is a Lie groupoid, its Lie algebroid can be identified with kernel of the source map. The source map is $Ns:N^G_H\to N^{G^0}_{H^0}$. Hence using proposition \ref{last min prop}, $dNs$ can be identified with $Nds:N^{TG}_{TH}\to N^{TG^0}_{TH^0}$. The kernel under this identification becomes $N^{\mathfrak{A}G}_{\mathfrak{A}H}$. Here we similarly identified $\mathfrak{A}G,\mathfrak{A}H$ with the kernel of the source map.
 
For the Lie algebroid of the deformation to the normal cone, one proceeds similarly. The Lie algebraoid of $\dnc(G,H)$ is the kernel of $d\dnc(s):T\dnc(G,H)\to T\dnc(G^0,H^0)$. The kernel of such map has to lie in $\ker(d\pi_\R).$ Hence one can instead only consider $$d\dnc(s):\ker(d\pi_R)\subseteq T\dnc(G,H)\to \ker(d\pi_\R)\subseteq T\dnc(G^0,H^0).$$ One then procceds exactly the same as for $\mathcal{N}^G_H$. 

 The third statement follows from the first and because the projection map onto the base
$$ \begin{tikzcd}N^G_H\arrow[d,shift right=0.85]\arrow[d, shift left=0.85]\arrow[r]& H\arrow[d,shift right=0.85]\arrow[d, shift left=0.85]\\ N^{G^0}_{H^0}\arrow[r]&H^0
\end{tikzcd}$$ is a submersive morphism of groupoids, hence the inverse image of the Lie subgroupoid $K$ is a Lie groupoid.
\end{proof}
From now on, for a Lie groupoid $G$ and a Lie subgroupoid $H$, we will use $\mathcal{N}^G_H$ to denote the space $N^G_H$ equipped with the structure of a Lie groupoid given by \Cref{Construction of dnc grouopids}.
\begin{rems}\label{sections of dnc algebr}\label{trivializing dnc Rn}\label{metric dnc E}\begin{enumerate}
 \item Let $E\to M$ be a vector bundle, $V\subseteq M$ a submanifold, $F\to V$ a subbundle of the restriction of $E$ to $V$. By \Cref{Construction of dnc grouopids}, the space $\dnc(E,F)$ is a vector bundle over $\dnc(M,V).$ Since a section of $\dnc(E,F)$ is determined by its values on the dense set $M\times \R^*$. It follows that $$\Gamma(\dnc(E,F))=\{X\in \Gamma(E\times \R):X_{|V\times\{0\}}\in \Gamma(F)\},$$ where $\Gamma$ denotes the set of global sections (continuous or smooth).
 \item Let $V=V_0+a\subseteq \R^n$ be an affine subspace where $V_0$ is the underlying vector space, $a\in \R^n$. Let $L$ be the orthogonal of $V_0$, $\pi_{V_0},\pi_L$ the orthogonal projections. The space $\dnc(\R^n,V)$ will be identified with $\R^{n+1}$ by the following map \begin{align*}
&\dnc(\R^n,V)\to \R^{n+1}\\
&(x,t)\to (a+\pi_{V_0}(x-a)+\frac{\pi_L(x-a)}{t},t),\quad t\neq 0\\
&(x,X,0)\to (x+X,0),
\end{align*}
where in the last identity we identified $N^{\R^n}_V$ with $L$.
\end{enumerate}
\end{rems}
\begin{examps}\label{examps dnc}\begin{enumerate}
\item If $M$ is a smooth manifold, then $$\dnc(M\times M,M)=M\times M\times \R^*\sqcup TM\times \{0\}\rightrightarrows M\times \R$$ is the tangent groupoid of Connes. He used it to give a short elegant proof of Atiyah Singer index theorem \cite{MR1303779}. The product law is given by $$(x,y,t)\cdot (y,z,t)=(x,z,t),\; (x,X,0)\cdot (x,Y,0)=(x,X+Y,0).$$
\item Let $L\subseteq G^0$ be a submanifold. Here we will calculate $\mathcal{N}^G_L$. Notice that $N^G_{L}$ is equal to $ \ker(ds)_{|L}\oplus N^{G^0}_L$. If $(Y,Z)\in  \ker(ds)_{|L}\oplus N^{G^0}_L$, then $s_{\mathcal{N}^G_L}(Y,Z)=ds(Y)+ds(Z)=Z$ by assumption on $Y$. Also $r_{\mathcal{N}^G_L}(Y,Z)=dr(Y)+dr(Z)=\natural(Y)+Z.$ Here we used the definition of the anchor map $\natural:=dr-ds$. It follows that the groupoid $\mathcal{N}^G_{L}\rightrightarrows N^{G^0}_L$ is equal to $$\{(X,Y,Z):l\in L,X,Z\in N^{G^0}_{L,l},Y\in \mathfrak{A}_lG, X=Y+\natural(Z)\},$$ with the structural maps $$r(X,Y,Z)=X,s(X,Y,Z)=Z.$$ Finally the product is given by $$(A,B,C)(C,D,E)=(A,B+D,E).$$ To see this notice that one has a natural map $\mathcal{N}^G_L\to \mathcal{N}^G_{G^0}.$ This is simply the quotient map. In the above identification this map sends $(X,Y,Z)\to Y$. By functoriality of the map, it has to be a morphism of groupoids, hence the product has to agree with the product of the adiabatic groupoid $N^G_{G^0}.$
\end{enumerate}
\end{examps}

\subsection{DNC iterated}\label{Dnc iterated sect}
Let $M$ be a smooth manifold, $V_0\subseteq M$ a submanifold, $V_{1}\subseteq \dnc(M,V_0)$ a submanifold. One defines $$\dnc^2(M,V_0,V_{1}):=\dnc(\dnc(M,V_0),V_{1}).$$ 
This space being a deformation space admits an $\R^*$-action that will be denoted by $\lambda^{(1)},$ and a projection map $\pi_\R^{(1)}:\dnc^2(M,V_{0},V_{1})\to \R$.

\bigskip

If $V_{1}$ is $\R^*$-invariant, then by \Cref{action dnc}, the group $\R^*$ acts on $\dnc^2(M,V_{0},V_{1})$. This action will be denoted by $\lambda^{(0)}$, furthermore the group $\left(\R^*\right)^2$ acts on $\dnc^2(M,V_0,V_1)$ by $\lambda^{(0)}\times \lambda^{(1)}.$

\bigskip

Let $\pi_\R:\dnc(M,V)\to \R$ be the projection constructed in \Cref{dnc section}. If $\pi_\R(V_1)$ is a point of $\R$, then the map $$\pi_\R^{(0,1)}:=\dnc(\pi_\R):\dnc^2(M,V_{0},V_{1})\to \dnc(\R,\pi_\R(V_{1}))=\R^2$$ is a smooth submersion, where  we identified $\dnc(\R,\pi_\R(V_1))$ with $\R^2$ using \cref{trivializing dnc Rn}.

If $V_{1}$ is furthermore $\R^*$-invariant (hence $\pi_\R(V_1)=\{0\}$), then one has for all $u,t\in\R^*$ 
\begin{equation}\label{relation projection and action}
\pi^{(0,1)}_\R\lambda^{(1)}_u=(\frac{\pi^{(0)}_\R}{u},u\pi^{(1)}_\R),\; \pi_\R^{(0,1)}\lambda^{(0)}_u=(u\pi^{(0)}_\R,\pi_\R^{(1)}),
\end{equation} where $\pi^{(0,1)}_\R=(\pi^{(0)}_{\R},\pi^{(1)}_\R)$.

\bigskip

By induction, given a sequence of submanifolds $$V_0\subseteq M,\;V_{1}\subseteq \dnc(M,V_0),\;V_{2}\subseteq \dnc^2(M,V_0,V_{1}),\cdots ,\;V_{k}\subseteq \dnc^{k}(M,V_0,\dots ,V_{k-1}).$$ We define the space $$\dnc^{k+1}(M,V_0,\cdots,V_k):=\dnc(\dnc^{k}(M,V_0,\cdots,V_{k-1}),V_{k}).$$

If for each $1\leq i\leq k$, $\pi^{(0,\dots,i-1)}_\R(V_i)$ is an affine subspace of $\R^i$ and $\pi^{(0,\dots,i-1)}_\R:V_i\to \pi^{(0,\dots,i-1)}_\R(V_i)$ is a submersion, then by \Cref{Functoriality of DNC}, the map $$\pi_\R^{(0,\cdots,k)}:=\dnc(\pi_\R^{(0,\cdots,k-1)}):\dnc^{k+1}(M,V_0,\cdots,V_k)\to \dnc(\R^k,\pi^{(0,\dots,k-1)}_\R(V_k))=\R^{k+1}$$ is a smooth submersion, where we identified $\dnc(\R^k,\pi_\R^{(0,\dots,k-1)}(V_k))$ with $\R^{k+1}$ using \cref{trivializing dnc Rn}.

If each $V_i$ is $(\R^*)^i$ invariant, then the space $\dnc^{k+1}(M,V_0,\cdots,V_k)$ admits $k+1$ pairwise commuting actions of $\R^*$-denoted $\lambda^{(k)},\dots,\lambda^{(0)}$.

\bigskip

Propositions \ref{Functoriality of DNC},\;\ref{DNC fibered product} and \Cref{Construction of dnc grouopids} have obvious extensions to $\dnc^k$. \begin{cor} If $G\rightrightarrows G^0$ is a Lie groupoid, $H_0\subseteq G$, $H_{1}\subseteq \dnc(G,H_0)$, $\dots ,\;H_{k}\subseteq \dnc^{k}(G,H_0,\dots ,H_{k-1}).$ are Lie subgroupoids, then $$\dnc^{k+1}(G,H_0,H_{1},\cdots,H_k)\rightrightarrows \dnc^{k+1}(G^0,H_0^0,\cdots,H_k^0)$$ is a Lie groupoid.\end{cor}
\subsection{Description of the symbol part}\label{prem sect prop}
The Carnot groupoid is a groupoid of the form $$M\times M\times \R^*\sqcup Q\times \{0\}\rightrightarrows M\times \R.$$ In order to describe the "symbol part" $Q\times \{0\}$ of this groupoid, we prove the following proposition in a slightly more general setting. 
\begin{prop}\label{prop exact seq}Let $G\rightrightarrows G^0$ be a Lie groupoid, $H\subseteq G$ a Lie subgroupoid which is a bundle of connected Lie groups such that $$(dr-ds)(T_hG)\subseteq T_{s(h)}H^0,\quad \forall h\in H.$$ Then \begin{enumerate}
\item the Lie groupoid $\mathcal{N}^G_H\rightrightarrows N^{G^0}_{H^0}$ is a bundle of Lie groups.
\item the Lie groupoid $\mathcal{N}^G_H\big|_{H^0}\rightrightarrows N^{G^0}_{H^0}$ is a bundle of abelian Lie groups which is isomorphic (as a bundle of Lie groups over $N^{G^0}_{H^0}$) to $\mathfrak{A}G/\mathfrak{A}H\times_{H^0}N^{G^0}_{H^0}.$
\item the Lie groupoid $\mathcal{N}^G_H\rightrightarrows N^{G^0}_{H^0}$ sits in an exact sequence of bundles of Lie groups over $N^{G^0}_{H^0}$ whose fiber at $(x_0,X_0)\in N^{G^0}_{H^0}$ is $$1\to\mathfrak{A}G_{x_0}/\mathfrak{A}H_{x_0}\to\left(\mathcal{N}^G_H\right)_{(x_0,X_0)}\to  H_{x_0}\to 1.$$ 
\end{enumerate}
Furthermore the action associated to this exact sequence of the Lie algebra $\mathfrak{A}H_{x_0}$ on the abelian group $\mathfrak{A}G_{x_0}/\mathfrak{A}H_{x_0}$ is as follows; if $X,Y\in \Gamma^\infty(\mathfrak{A}G)$ such that $X_{|H^0}\in \Gamma^\infty(\mathfrak{A}H)$, then by our assumption, $$[X,Y](x_0)\mod \mathfrak{A}H_{x_0}$$ only depends on $X(x_0)\in \mathfrak{A}H_{x_0}$ and $Y(x_0)\mod \mathfrak{A}H_{x_0}\in \mathfrak{A}G_{x_0}/\mathfrak{A}H_{x_0}$. In particular the above exact sequence is central if and only if this action is trivial.
\end{prop}
\begin{proof}
\begin{enumerate}
\item The condition $(dr-ds)(T_hG)\subseteq T_{s(h)}H^0$ can be restated as the equality of the maps $Ns,Nr:T_hG/T_hH\to T_{s(h)}G^0/T_{s(h)}H^0$. Those two maps are the source and the target maps of  the Lie groupoid  $\mathcal{N}^G_H=\sqcup_{h\in H}T_hG/T_hH\rightrightarrows N^{G^0}_{H^0}$. By assumption, they coincide which means that $\mathcal{N}^G_H\rightrightarrows N^{G^0}_{H^0}$ is a bundle of Lie groups.

\item If $X\subseteq Y\subseteq Z$ are manifolds, then $N^Z_Y\big|_X=N^Z_X/N^Y_X$. It follows that $\mathcal{N}^G_{H}\big|_{H^0}$ is the surjective image by a groupoid morphism of the Lie groupoid $\mathcal{N}^G_{H^0}$ with kernel $N^{H}_{H^0}.$ One has $$\mathcal{N}^G_{H^0}=\{(X,Y,Z):X,Z\in N^{G^0}_{H^0},Y\in \mathfrak{A}G/\mathfrak{A}H, Z=X+\natural(Y)\}\rightrightarrows N^{G^0}_{H^0}.$$ By assumption, the map $\natural:\mathfrak{A}G/\mathfrak{A}H\to \mathcal{N}^{G^0}_{H^0}$ is the zero map. Hence $\mathcal{N}^G_{H^0}\rightrightarrows N^{G^0}_{H^0}$ is a bundle of abelian Lie groups, hence $\mathcal{N}^G_{H}\big|_{H^0}\rightrightarrows N^{G^0}_{H^0}$ as well.
\item the exact sequence is the natural sequence $$ \mathcal{N}^G_H|_{H^0}\to \mathcal{N}^G_H\to H.$$ \begin{enumerate}
\item exactness at $\mathfrak{A}G_{x_0}/\mathfrak{A}H_{x_0}$
 is clear, because $\mathcal{N}^G_{H}\big|_{H^0}$ is a subgroupoid of $\mathcal{N}^G_H$
 \item  exactness at $\left(\mathcal{N}^G_H\right)_{(x_0,X_0)}$ follows directly from the definitions.
 \item the map $s:G\to G^0$ is a submersion, hence exactness at $H_{x_0}$.
 \end{enumerate}
 Let us prove that $[X,Y]$ only depends on $X(x_0)$ and $Y(x_0)$, where $X,Y\in \Gamma^\infty(\mathfrak{A}G)$  such that $X_{|H^0}\in \Gamma^\infty(\mathfrak{A}H).$\begin{itemize}
 \item If $Y$ vanishes at $x_0$, then locally it can be written as the sum of sections of the form $fZ$, where $f:M\to \R$ vanishes at $x_0$ and $Z\in \Gamma^\infty(\mathfrak{A}G)$. One has $$[X,fZ]=f(x_0)[X,Z](x_0)+df_{x_0}(\natural(X(x_0)))Z(x_0)=0,$$ because $X(x_0)\in \mathfrak{A}H_{x_0}$ and $H$ is a bundle of Lie groups, hence $\natural (X(x_0))=0$.
 \item If $Y_{|H^0}\in \Gamma^\infty(\mathfrak{A}H)$, then $[X,Y](x_0)\in \mathfrak{A}H_{x_0}$ because the Lie bracket computation could be carried out inside $\mathfrak{A}H$.
 \item If $X$ vanishes at $x_0$, then $dX_{x_0}:T_{x_0}G^0\to \mathfrak{A}_{x_0}G$ is well defined. It is well known that $[X,Y](x_0)=-dX_{x_0}(\natural(Y(x_0))).$ This formula can be proved locally by writing $X$ as sum of $fZ$. The condition $X_{|H^0}\in \Gamma^\infty(\mathfrak{A}H)$ implies that $dX_{x_0}(T_{x_0}H^0)\subseteq \mathfrak{A}H^0$. The assumption on $dr-ds$ implies that $\natural (Y(x_0))\in T_{x_0}H^0$, hence $[X,Y](x_0)\in \mathfrak{A}H_{x_0}.$
 \end{itemize}
\end{enumerate}
That this is the action associated to the abelian extension of $\left(\mathcal{N}^{G}_{H}\right)_{(x_0,X_0)}$ is then clear.
\end{proof}

\section{The case of a single subbundle}\label{sect case 1}
\subsection{The product structure of the symbol part}\label{M single H examp}
Let $M$ be a smooth manifold, $H\subseteq \mathcal{N}^{M\times M}_M=TM$ a subbundle. In this section we prove \Cref{main thm H single}, which proves the claim made in the introduction (at least on the algebraic level) that the fiber of the groupoid $\dnc^2(M\times M,M,H\times \{0\})\rightrightarrows \dnc^2(M,M,M\times\{0\})=M\times \R^2$ over $M\times \{1\}\times \R$ is equal to the groupoid constructed in \cite{Choi:2015aa,Choi:2017aa,Ponge:2017aa,Erp:2016aa}. In \Cref{sect H bundle 1 case}, we will write local charts which will prove that in fact the fiber is equal as a smooth manifold to the one constructed in  \cite{Choi:2015aa,Choi:2017aa,Ponge:2017aa,Erp:2016aa}. 

Before stating the theorem, let us recall the constuction of the Levi form $\mathcal{L}$ : the map \begin{align}\label{dfn : levi}
\Gamma^\infty(H)\times \Gamma^\infty(H)\to \Gamma^\infty(TM/H),\quad (X,Y)\to [X,Y]\mod H
\end{align}
is $C^\infty(M)$-linear because \begin{align*}
[fX,Y]=f[X,Y]-Xf Y=f[X,Y]\mod H.
\end{align*} Hence it comes from an anti symmetric bilinear bundle map $\mathcal{L}:H\times H\to TM/H$.
\begin{theorem}\label{main thm H single}
The groupoid $\mathcal{N}^{\dnc(M\times M,M)}_{H\times\{0\}}\rightrightarrows N^{M\times \R}_{M\times\{0\}}=M\times \R$ is a bundle of Lie groups. It is isomorphic to the bundle of Lie groups $H\oplus TM/H\times\R\rightrightarrows M\times \R$ equipped with the group law $$(h,n,t)\cdot (h',n',t)=\left(h+h',n+n'+\frac{t}{2}\mathcal{L}(h,h'),t\right).$$
\end{theorem} 
\begin{proof}
First we apply \Cref{prop exact seq} to $\dnc(M\times M,M)\rightrightarrows M\times\R$ and $H\times\{0\}\rightrightarrows M\times \{0\}$. Let us check the condition of \Cref{prop exact seq} and the triviality of the action. \begin{itemize}
 \item Since $\pi_\R\circ r=\pi_\R\circ s$, the condition of \Cref{prop exact seq} is satisfied.
  
\item the triviality of the action is immediate to check. If $X$ is a section of $TM$ over $M\times \R$ which vanishes on $M\times \{0\}$, $Y$ is a section of $TM$ over $M\times \R$ which vanishes on $M\times \{0\}$ and whose $\partial_t$-derivative on $M\times\{0\}$ is in $H$, then the vector field $[X,Y]$ vanishes over $M\times \{0\}$.
\end{itemize} 

The central exact sequence of bundles of Lie groups over $(x_0,t_0)\in N^{M\times\R}_{M\times \{0\}}=M\times \R$ given by \Cref{prop exact seq} is then equal to $$1\to T_{x_0}M/H_{x_0}\to \left(\mathcal{N}^{\dnc(M\times M,M)}_{H\times\{0\}}\right)_{(x_0,t_0)}\to H_{x_0}\to 1.$$

There exists a quite natural section of this exact sequence: let $h\in H_{x_0}$, $f:\R\to M$ any smooth function such that $f(0)=x_0$, $f'(0)=h$ and $f'(t)\in H_{f(t)} \;\forall t$, \begin{align*}
&\sigma_{x_0,t_0}(h,\cdot):\R\to \dnc(M\times M,M)\\
 &\sigma_{x_0,t_0}(h,u)=(f(t_0u),x_0,t_0u) &\text{if}\; t_0u\neq 0\\
 &\sigma_{x_0,t_0}(h,0)=(x_0,h,0)&\text{if}\; t_0u=0
\end{align*}
One then sees immediately that the map \begin{align*}\mathfrak{S}_{x_0,t_0}:H_{x_0}&\to  \left(\mathcal{N}^{\dnc(M\times M,M)}_{H\times\{0\}}\right)_{(x_0,t_0)}\\
 h&\to\left( \frac{\partial}{\partial u}\big|_{u=0}\sigma_{x,t}(h,u)\mod T_{(x_0,h,0)}(H\times \{0\})\right)
\end{align*} is well defined (i.e, doesn't depend on the choice of $f$) and is a section of the above exact sequence.

\bigskip

The map $\mathfrak{S}_{x_0,t_0}$ is not a group homomorphism. For $h_1,h_2\in H_{x_0}$, we have \begin{equation*}
 \mathfrak{S}_{x_0,t_0}(h_1)\mathfrak{S}_{x_0,t_0}(h_2)\mathfrak{S}_{x_0,t_0}(-h_1-h_2)=\frac{t_0}{2}\mathcal{L}(h_1,h_2)\in T_{x_0}M/H_{x_0}.
\end{equation*}
This follows from the definition of $\mathcal{L}.$ See \cref{dfn : levi}.
\end{proof}
\begin{cor}
The fiber of the groupoid $$\dnc^2(M\times M,M,H\times \{0\})\rightrightarrows M\times \R^2$$ over $M\times \{1\}\times \R$ is equal to (as an algebraic groupoid) to $$M\times M\times \R^*\sqcup H\oplus TM/H\times \{0\}\rightrightarrows M\times \R,$$ where the groupoid structure on $M\times M\times \R^*$ is the pair groupoid, and on $H\oplus TM/H$ is the bundle of nilpotent Lie groups $$(h,n)\cdot (h',n')=\left(h+h',n+n'+\frac{1}{2}\mathcal{L}(h,h')\right).$$
\end{cor}
Since $H$ is $\R^*$ invariant, by \Cref{Dnc iterated sect} we have two group actions $\lambda^1$, $\lambda^0$ of $\R^*$ on $\mathcal{N}^{\dnc(M\times M,M)}_{H\times\{0\}}$. Under the above identification the two actions $\lambda^1$ and $\lambda^0$ become $$\lambda^0_u(h,n,t)=(\frac{h}{u},\frac{n}{u},ut),\quad \lambda^1_u(h,n,t)=(h,\frac{n}{u},\frac{t}{u}).$$
\subsection{Local charts for $N^{\dnc(M,V)}_{H\times\{0\}}$}\label{sect H bundle 1 case}
In this section the development done in section 2.1 at the level of Lie algebroids is done in parallel at the level of local charts. This is more general as it applies to $N^{\dnc(M,V)}_{H\times\{0\}}$ which is in general only a smooth manifold.

Let $M$ be a smooth manifold, $V$ a submanifold, $H\subseteq N^M_V$ a smooth subbundle, $\mathcal{H}$ the lift of $H$ to $TM$. In other words $\mathcal{H}$ is a subbundle of the restriction of $TM$ to $V$ such that $TV\subseteq \mathcal{H}$ and $H=\mathcal{H}/TV$. In this section we give an alternate description of the fiber $\left(\pi^{(0,1)}\right)^{-1}(\{1\}\times \R)$ of the space $\dnc^2(M,V,H\times \{0\})$.

\begin{dfn}Let $\tilde{N}^M_{V,H}$ the set of smooth functions $f:\R\to M$ such that $ f(0)\in V$ and $f'(0)\in \mathcal{H}_{f(0)}$. 

Let $N^M_{V,H}$ be the quotient of $\tilde{N}^M_{V,H}$ by the equivalence relation where $f,g\in \tilde{N}^M_{V,H}$ are equivalent if and only if \begin{enumerate}
\item $f(0)=g(0)$
\item  $f'(0)-g'(0)\in T_{f(0)}V$.
\item  for every smooth function $l:M\to \R$ which vanishes on $V$ and whose derivative $dl$ vanishes on $\mathcal{H}$, one has $(l\circ f)''(0)=(l\circ g)''(0)$.
\end{enumerate}\end{dfn}
Let $\pi_\R:\dnc(M,V)\to \R$ be the projection. Since $\pi_\R(H)=0$, the map $N\pi_\R:N^{\dnc(M,V)}_H\to N^\R_{0}=\R$ is well defined. We claim that the set $N^M_{V,H}$ is in a natural bijection with $\left(N\pi_\R\right)^{-1}(1).$ To see this let $f\in \tilde{{N}}^M_{V,H}$. Since $f(0)\in V$, the function \begin{align*}
\dnc(f):\R\to \dnc(M,V),\quad t\to (f(t),t),\text{if}\,t\neq 0,\; 0\to (f'(0),0)
\end{align*}
is smooth. In the previous formula instead of the domain being $\dnc(\R,0),$ we replace the domain with $\R$ using the inclusion \begin{align*}
\R\to \dnc(\R,0)\\
t\to (t,t)\\
0\to (1,0).
\end{align*} Since $f'(0)\in H$ it follows that $
 \dnc^2(f):\R\to \dnc^2(M,V,H)$ is a welll defined smooth map. Its value at zero is an element in $N^{\dnc(M,V)}_H$ which is clearly in $\left(N\pi_\R\right)^{-1}(1)$. \begin{prop}the map \begin{align*}\label{map description}
\beta:N^M_{V,H}\to \left(N\pi_\R\right)^{-1}(\{1\}),\quad [f]\to [\dnc(f)]
\end{align*}
is a well defined bijection
 \end{prop}
 Let us remark that the map $\beta$ is not a linear map and in fact the space $N^M_{V,H}$ is not a vector bundle. 
 \begin{proof}
   In \Cref{dnc section}, two types of functions on $\dnc(M,V)$ were described which generate the ring of smooth functions on $\dnc(M,V).$ By regarding each type we see that for two functions $f,g\in \tilde{N}^M_{V,H}$, the classes in   $N^{\dnc(M,V)}_{H}$ of $\dnc(f)$ and $\dnc(g)$ are equal if and only if the classes of $f$ and $g$ are equal in $N^M_{V,H}$. Hence $\beta$ is well defined and injective. Surjectivity follows by looking at a local chart as described below.
\end{proof}   
    Let $\psi:N^M_V\to M$ be a tubular neighbourhood embedding, $L:H\oplus N^M_V/H\to N^M_V$ a linear isomorphism given by the choice of a complementary subbundle of $H$ inside $N^M_V$, $\phi=\psi\circ L$.

By the local charts descriped in \Cref{dnc section}, the following is a local chart for $\dnc(M,V)$: \begin{align*}
&\tilde{\phi}:H\oplus N^M_V/H\times \R\to \dnc(M,V)\\
&(h,n,t)\to (\phi(th,tn),t),\quad t\neq 0\\
&(h,n,0)\to (L(h,n,0),0).
\end{align*}
Therefore the following is a local chart for $\dnc^2(M,V,H\times\{0\})$
 \begin{align*}
&H\oplus N^M_V/H\times \R\times\R\to \dnc^2(M,V,H\times\{0\})\\
&(h,n,t,u)\to (\phi(uth,u^2tn),ut,u)\in M\times \R^*\times \R^*\quad t\neq 0,u\neq 0\\
&(h,n,0,u)\to (L(h,un),0,u)\in N^M_V\times \{0\}\times \{u\},\quad u\neq 0\\
&(h,n,t,0)\to (h,n,t,0)\in N^{\dnc(M,V)}_H\times\{0\},
\end{align*}
where in the last identity we identified $N^{\dnc(M,V)}_H$ with  $H\oplus N^M_V\oplus \R$ using $\tilde{\phi}.$ In this local picture, $\pi^{(0,1)}$ is the projection $(h,n,t,u)\to (t,u).$

\bigskip

Let $$\dnc_H(M,V):=M\times \R^*\sqcup N^M_{V,H}\times\{0\}.$$  We equip $\dnc_H(M,V)$ with a smooth structure by identifying it with $\left(\pi_\R^{(0,1)}\right)^{-1}(\{1\}\times \R)$ using the map $\beta$. Its local charts are hence given by \begin{align*}
H\oplus N^M_V/H\times \R\to \dnc_H(M,V)\\
(h,n,u)\to (\phi(uh,u^2n),u),\quad u\neq 0\\
(h,n,0)\to ([t\mapsto \phi(L(th,t^2n))],0).
\end{align*}
The space $\dnc_H(M,V)$ is called \textit{the deformation to the normal cone of $M$ along $V$ with weight $H$.}

\begin{rem}\label{rem all fibers are same}All the other fibers $\left(\pi^{(0,1)}_\R\right)^{-1}(\{t\}\times \R)$ for $t\neq 0$ are isomorphic to $\left(\pi^{(0,1)}_\R\right)^{-1}(\{1\}\times \R)$ by a rescaling in the $u$-variable. The fiber $\left(\pi^{(0,1)}_\R\right)^{-1}(\{0\}\times \R)$ is equal to $\dnc(N^M_V,H).$ In particular the space $\dnc^2(M,V,H)$ should be seen as a deformation of the space $\dnc_H(M,V)$ to the simpler space $\dnc(N^M_V,H).$
\end{rem}
Since $H$ is $\R^*$-invariant, by \Cref{Dnc iterated sect} it follows that there is an $(\R^*)^2$ action on $\dnc^2(M,V,H\times\{0\}).$ It follows from \Cref{relation projection and action} in \Cref{Dnc iterated sect} that $\left(\pi_\R^{(0,1)}\right)^{-1}(\{1\}\times \R)$ is invariant under the diagonal $\lambda^{(1)}_u\lambda^{(0)}_u$. This action is described by $u\cdot (x,t)=(x,tu)$ and $u\cdot ([f],0)=([f(\frac{\cdot}{u})],0)$ for $f\in \tilde{N^M_{V,H}}.$
 \begin{cor}\label{naturality dnc single H}Let $(M,V), (M',V')$ be smooth manifold pairs, $H\subseteq N^M_V, H'\subseteq N^{M'}_{V'}$ subbundles, $g:M\to M'$ a smooth map  such that  $g(V)\subseteq V'$ 
and $dg(H)\subseteq H'$. Then the maps \begin{itemize}
\item $\begin{aligned}[t]
 Ng:N^M_{V,H}\to N^{M'}_{V',H'}\quad [f]\to [g\circ f]
\end{aligned} $
\item $\begin{aligned}[t]
\dnc(g):\dnc_H(M,V)\to \dnc_{H'}(M',V')\\
(x,t)\to (g(x),t)\\
([f],0)\to ([g\circ f],0)
\end{aligned}$
\end{itemize}
are well defined and smooth.
\end{cor}
\begin{proof}
This is a corollary of \Cref{Functoriality of DNC} applied twice and the identification of $\dnc_H(M,V)$ with $\left(\pi^{(0,1)}_\R\right)^{-1}(\{1\}\times \R)\subseteq \dnc^2(M,V,H\times\{0\})$.
\end{proof}
\begin{prop}\label{fibered product dnc commute single H}Let $M_1, M_2, M$ be manifolds, $V_i\subseteq M_i,V\subseteq M$ submanifolds, $H_i\subseteq N^{M_i}_{V_i}$, $H\subseteq N^M_V$ vector subbundles, $f_i:M_i\to M$ smooth maps such that \begin{enumerate}
\item $f_i(V_i)\subseteq V$
\item the maps $f_i:M_i\to M$ are transverse
\item the maps $f_i|V:V_i\to V$ are transverse
\item  $H=df_1(H_1)+df_2(H_2)$,
\end{enumerate} then \begin{enumerate}
\item the maps $\dnc(f_i):\dnc_{H_i}(M_i,V_i)\to \dnc_H(M,V)$ are transverse.
\item  the natural map $$\dnc_{H_1\times_H H_2}(M_1\times_M M_2,V_1\times
_V V_2)\to\dnc_{H_1}(M_1,V_1)\times_{\dnc_H(M,V)}\dnc_{H_2}(M_2,V_2)$$  is a diffeomorphism.
\end{enumerate}
\end{prop}
\begin{proof}
This is a corollary of \Cref{DNC fibered product} applied twice and the identification of $\dnc_H(M,V)$ with $\left(\pi^{(0,1)}_\R\right)^{-1}(\{1\}\times \R)\subseteq \dnc^2(M,V,H\times\{0\})$. 
\end{proof}
\begin{theorem}\label{dnc groupoid is a groupoid single H}
Let $G\rightrightarrows G^0$ be a groupoid, $G'\rightrightarrows G^{'0}$ a subgroupoid, $H\subseteq \mathcal{N}^G_{G'}$ a $\mathcal{VB}$-subgroupoid \cite{MR941624,MR896907}. Then \begin{enumerate}
\item the space $\mathcal{N}^G_{G',H}\rightrightarrows N^{G^0}_{G^{'0},H^0}$ is a Lie groupoid whose algebroid is equal to $N^{\mathfrak{A}G}_{\mathfrak{A}G',\mathfrak{A}H}.$
\item  the space $\dnc_H(G,G')\rightrightarrows\dnc_{H^0}(G^0,G^{'0})$ is a Lie groupoid whose Lie algebroid is equal to $\dnc_{\mathfrak{A}H}(\mathfrak{A}G,\mathfrak{A}G').$
\end{enumerate}
\end{theorem}
\begin{proof}
This is a corollary of \Cref{naturality dnc single H} and \Cref{fibered product dnc commute single H}.
\end{proof}
\begin{examp}\label{example H foliation single}Let $F\subseteq TM$ be an integrable subundle. We regard the foliation groupoid $\mathcal{G}(M,F)\rightrightarrows M$ as an immersed subgroupoid of $M\times M\rightrightarrows M$ by the map $$(x,[\gamma],y)\to (x,y).$$ This map is not injective but the Lie groupoid $\dnc(M\times M,\mathcal{G}(M,F))\rightrightarrows M\times \R$ is still well defined by \cref{rem nonhaus}.  Its underlying manifold is a second countable locally Hausdorff manifold.

\bigskip

 The vector bundle $TM/F$ will be denoted by $\nu(F).$  If $\gamma:[0,1]\to M$ is path tangent to the leaves, then its holonomy defines a map $d\gamma:\nu(F)_{\gamma(0)}\to \nu(F)_{\gamma(1)}$. One then sees that the groupoid $$\mathcal{N}^{M\times M}_{\mathcal{G}(M,F)}=\{(x,[\gamma],y,X):(x,[\gamma],y)\in \mathcal{G}(M,F),X\in \nu(F)_{y}\}\rightrightarrows M.$$The product is then given by $$(x,[\gamma],y,X)\cdot (y,[\gamma'],z,Y)=(x,[\gamma\gamma'],z,d\gamma'(X)+Z).$$ Let  $H\subseteq \nu(F)$ be a holonomy invariant subbundle, i.e such that for any leafwise path $\gamma:[0,1]\to M$, one has $d\gamma(H_{\gamma(0)})=H_{\gamma(1)}$. It follows that $$L:=\{(x,[\gamma],y,X)\in \mathcal{N}^{M\times M}_{\mathcal{G}(M,F)}:X\in H_y\}\subseteq \mathcal{N}^{M\times M}_{\mathcal{G}(M,F)}$$ is a Lie subgroupoid. The groupoid $$\mathcal{N}^{M\times M}_{\mathcal{G}(M,F),L}=\{(x,[\gamma],y,X,Y):X\in H_y,Y\in \nu(F)_{y}\}\rightrightarrows M$$ has then the groupoid law $$(x,[\gamma],y,X,Y)\cdot (y,[\gamma'],z,X',Y')=(x,[\gamma\gamma'],z,d\gamma'(X)+X',d\gamma'(Y)+Y'+\frac{1}{2}\mathcal{L}(d\gamma'(X),X')),$$where $\mathcal{L}:H\times H\to v(F)/H$ is a Levi form defined similarly to the one defined in \Cref{M single H examp}.
\end{examp}

\section{Carnot Groupoid}\label{Heisenberg groupoid sect}
A more general groupoid will be constructed starting from the following data: Let $M$ be a smooth manifold, $0=H^0\subseteq H^1\subseteq \cdots\subseteq H^{k+1}=TM$ be vector bundles such that $$[\Gamma^\infty(H^i),\Gamma^\infty(H^j)]\subseteq \Gamma^\infty(H^{i+j}),$$ where $H^{i}=TM$ for $i>k$. We will calculate the Lie algebroid of this groupoid and hence show that it is equal to the groupoid constructed in \cite{MR2417549,Choi:2015aa,Choi:2017aa,Ponge:2017aa,Erp:2016aa}. See remark 3.2.3 for more details.

\bigskip

 Since $[\Gamma^\infty(H^i),\Gamma^\infty(H^i)]\subseteq \Gamma^\infty(H^{i+j})$, it follows that the map \begin{align*}
\Gamma^\infty(H^i/H^{i-1})\times \Gamma^\infty(H^{j}/H^{j-1})\to \Gamma^\infty(H^{i+j}/H^{i+j-1})\\
(X,Y)\to [X,Y]\mod \Gamma^\infty(H^{i+j-1}) 
\end{align*} is a $C^\infty(M)$-bilinear map, hence it comes from an antisymmetric bilinear map $$\mathcal{L}:H^{i}/H^{i-1}\times H^{j}/H^{j-1}\to H^{i+j}/H^{i+j-1}.$$

For each $a\in M$, the map $\mathcal{L}$ defines the structure of a Lie algebra on $\mathcal{G}(H^\cdot)_a:=\oplus_{i}H^i_a/H^{i+1}_a$ by $$[X,Y]=\mathcal{L}(X,Y),\quad \text{for}\; X\in H^i_a/H^{i+1}_a,Y\in H^j_a/H^{j+1}_a.$$ By Baker–Campbell–Hausdorff formula, the vector space $\mathcal{G}(H^\cdot)_a$ admits the structure of a nilpotent Lie group. It is clear that the structure of group is $C^\infty$ in $a$, hence $\mathcal{G}(H^\cdot)$ is a bundle of nilpotent Lie groups. We will define a Lie groupoid denoted by $\dnc_{H^\cdot}(M\times M,M)$ by induction on $k$ whose underlying set is equal to $$M\times M\times \R^*\sqcup \mathcal{G}(H^\cdot)\times \{0\}.$$ and whose Lie algebroid is equal to $$\Gamma^\infty(\mathfrak{A}_{H^\cdot})=\{X\in \Gamma^\infty(TM\times \R):\partial_t^iX_{|t=0}\in \Gamma^\infty(H^i)\forall i\geq 0\}$$For $k=1$, this is just $\dnc_{H^1}(M\times M,M)\rightrightarrows M\times \R$ defined in \Cref{sect H bundle 1 case}. By induction assuming it is defined for $k-1$, that is the Lie groupoid \begin{align*}
 \dnc_{H^1,\dots,H^{k-1}}(M\times M,M)&=M\times M\times \R^*\sqcup \mathcal{G}(H^1,\dots,H^{k-1})\times \{0\}\\&=M\times M\times \R^*\sqcup H^1\oplus H^2/H^1\oplus \dots\oplus TM/H^{k-1}\times\{0\}
\end{align*} is well defined. The subset $H^1\oplus H^2/H^1\oplus \cdots \oplus H^k/H^{k-1}$ is a Lie subgroupoid of $\mathcal{G}(H^1,\dots,H^{k-1})$ precisely because $$[\Gamma^\infty(H^i),\Gamma^\infty(H^j)]\subseteq \Gamma^\infty(H^{k}),\quad i+j=k.$$ Therefore the space \begin{align*}
 \dnc(\dnc_{H^1,\dots,H^{k-1}}(M\times M,M)&,H^1\oplus H^2/H^1\oplus \cdots \oplus H^k/H^{k-1}\times \{0\})\\&\rightrightarrows \dnc(M\times \R,M\times \{0\})=M\times \R^2
\end{align*} is a Lie groupoid, where we used \Cref{trivializing dnc Rn}. The Lie algebroid of this groupoid is then $$\dnc(\mathfrak{A}_{H^1,\dots,H^{k-1}},H^1\oplus \cdots \oplus H^k/H^{k-1})$$

Using \Cref{sections of dnc algebr}, we get that the space of sections of this algebroid is then equal to \begin{align*}
\Gamma^\infty(\dnc(&\mathfrak{A}_{H^1,\dots,H^{k-1}},H^1\oplus \cdots \oplus H^k/H^{k-1})\\&=\{X\in \Gamma^\infty(TM\times \R\times \R):\partial^i_t(X)(0,u)\in \Gamma^\infty(H^i)\;\forall 0\leq i\leq k-1,u\in \R\\ &\&\;\partial^{k}_t(X)(0,0)\in \Gamma^\infty(H^k)\}.
\end{align*}

We define $\dnc_{H^1,\cdots,H^k}(M\times M,M)$ as the fiber of $\dnc(\dnc_{H^1,\dots,H^{k-1}}(M\times M,M),H^1\oplus H^2/H^1\oplus \cdots \oplus H^k/H^{k-1}\times \{0\})$ over $M\times \{1\}\times \R$. This is clearly a Lie groupoid.

\bigskip

It follows from the above description of $\Gamma^\infty(\dnc(\mathfrak{A}_{H^1,\dots,H^{k-1}},H^1\oplus \cdots \oplus H^k/H^{k-1}))$ by restricting to the diagonal we get that if $$X\in \Gamma^\infty(\dnc(\mathfrak{A}_{H^1,\dots,H^{k-1}},H^1\oplus \cdots \oplus H^k/H^{k-1}))$$ then $\partial_t^iX(0,0)\in\Gamma^\infty(H^i)$ for all $0\leq i\leq k$, where we used that $X(0,u)=0$. This fnishes the induction, and proves that Lie algebroid  of $\dnc_{H^\cdot}(M\times M,M)$ is equal to $\mathfrak{A}_{H^\cdot}$. Hence we proved the following\begin{theorem}
The Lie groupoid $\dnc_{H^\cdot}(M\times M,M)$ is the same as the groupoid constructed in  \cite{Choi:2015aa,Choi:2017aa,Ponge:2017aa,Erp:2016aa}.
\end{theorem}

\begin{rems}\begin{enumerate}
\item In \cite{Erp:2016aa}, a more general case is regarded where starting from a groupoid $G$, subbundles $H^1\subseteq \dots\subseteq H^r=\mathfrak{A}G$ such that $[\Gamma^\infty(H^i),\Gamma^\infty(H^j)]\subseteq \Gamma^\infty(H^{i+j})$ they construct a groupoid $\dnc_{H^\cdot}(G,G^0)$. It is clear that the above construction works equally well for this case with only notational changes. The advantage of our approach is that we can do the more general case of a groupoid inside another without any extra difficulty.
\item The groupoid \begin{align*}
 \dnc^{k+1}(M\times &M,M,H^1\times\{0\},\dots, H^1\oplus\dots \oplus H^{k}/H^{k-1}\times \R^{k-1}\times \{0\})\\&\rightrightarrows \dnc^{k+1}(M,M\times \{0\},\dots,M\times \R^{k-1}\times \{0\})=M\times \R^{k+1}.
\end{align*}is a Lie groupoid which contains the `deformations in all the directions'. This groupoid admits an $\left(\R^{*}\right)^{k+1}$ action as in \Cref{Dnc iterated sect}. The fiber over $(1,\dots,1,0)$ is then equal to $\dnc_{H^1,\dots, H^k}(M\times M,M)$. The action $\R^*$ defined on $\dnc_{H^1,\dots,H^k}(M\times M,M)$ defined in \cite{Erp:2016aa} is then just the diagonal action of $\R^{k+1}$ which by induction is easily seen to preserve the fiber $(1,\dots ,1,0).$

For example, in the case $k=2$, this gives \begin{align*}
 \dnc^3&(M\times M,H^1\times\{0\},H^1\oplus H^2/H^1\times \R)=M\times M\times \R^*\times \R^*\times \R^*\\&\sqcup TM\times \{0\}\times \R^*\times \R^*\sqcup H^1\oplus TM/H^1\times \R\times\{0\}\times \R^*\\&\sqcup H^1\oplus H^2/H^1\oplus TM/H^2\times\R\times\R\times\{0\}
\end{align*}
Let us remark that the subgroupoid $H^1\oplus H^2/H^1\oplus TM/H^2\times\R\times\R\times\{0\}$ is not trivial as a groupoid, it has a structure \begin{align*}
(h_1,h_2,h_3,t,u,0)\cdot (k_1,k_2,k_3,t,u,0)=\bigg(h_1+k_1,h_2+k_2+\frac{t}{2}[h_1,k_1],\\h_3+k_3+\frac{tu}{2}\left([h_1,k_2]+[h_2,k_1]\right)+\frac{t^2u}{12}\left([h_1,[h_1,k_1]]+[k_1,[k_1,h_1]]\right),t,u,0\bigg)
\end{align*}
Similarly for $M\times \{0\}\times \R^*\times \R^*$ and $H^1\oplus TM/H^1\times \R\times\{0\}\times \R^*$.
\item The existence of the Lie groupoid $\dnc_{H^\cdot}(M\times M,M)$ follows from Debord's result on integrability of Lie algebroids \cite{MR1751671}. Debord's result applies to the Lie algebroid of $\dnc_{H^\cdot}(M\times M,M)$. It shows that there exists a unique minimal Lie goupoid integrating the Lie algebroid of $\dnc_{H^\cdot}(M\times M,M)$. Minimal in the sense that any other Lie groupoid projects by a submersion morphism of groupoids into it.
\item The groupoid $\dnc_{H^\cdot}(M\times M,M)$ constructed above is the minimal groupoid integrating its Lie algebroid. The maximal Lie groupoid is the groupoid $$\tilde{M}\times_{\pi_1(M)}\tilde{M}\times\R^*\sqcup H^1\oplus \dots\oplus H^k/H^{k-1}\times\{0\},$$ where $\tilde{M}\times_{\pi_1(M)}\tilde{M}$ is the Poincaré groupoid.
\end{enumerate}
\end{rems}

\begin{examp}\label{examp foliation multipleHeisen}
Following the notation of \Cref{example H foliation single}. Let $F$ be a foliation, $H^1\subseteq\cdots\subseteq H^{k+1}=\nu(F)$ subbundles such that if $X\in \Gamma^\infty(H^i)$ and $Y\in \Gamma^\infty(H^j)$, then $$[X,Y]\in \Gamma^\infty(H^{i+j}).$$ with the convention $H^{s}=\nu(F)$ for $s>k$ and such that if $i\in\{1,\dots,k\}$, $\gamma:[0,1]\to M$ a path tangent to the leaves, then $d\gamma H^i_{\gamma(0)}=H^i_{\gamma(1)}$. In \Cref{example H foliation single}, we defined the groupoid $\dnc_{L^1}(M\times M,\mathcal{G}(M,F))$. We can by an induction, similar to the above, construct the groupoid $\dnc_{H^\cdot}(M\times M,\mathcal{G}(M,F))$.
 \end{examp}
\paragraph{Quotient of Lie groupoids}\label{para dfn quotient of Lie groupoids} Let $G\rightrightarrows G^0$ be a Lie groupoid, $H\subseteq G$ a Lie subgroupoid. The Lie groupoid $H$ acts on the smooth manifold $G_{H^0}$ by right translation. This action is clearly free. The action is proper if $H$ is closed in the pullback of $G$ by $H^0\subseteq G^0$. In this case, by \cite[section 5.9.5]{MR0219078}, the quotient space $G_{H^0}/H$ is a smooth manifold, that will be denoted by $G/H$.

\begin{examp}\begin{enumerate} \label{examp quotient dnc}\label{examp Higson}
\item If $V$ is a submanifold of $M$, then $\dnc(V\times V,V)$ is a Lie subgroupoid of $\dnc(M\times M,M)$. It is clear that the quotient space is equal to $$\dnc(M\times M,M)/\dnc(V\times V,V)=\dnc(M,V).$$

 \item Let $V\subseteq M$ a smooth submanifold such that $H^i\cap TV$ is of locally of finite rank. It is then clear that $[\Gamma^\infty(H^i\cap TV),\Gamma^\infty(H^j\cap TV)]\subseteq \Gamma^\infty(H^{i+j}\cap TV)$. Let $G(H^\cdot)$ the bundle of nilpotent Lie groups $\oplus_{i}Hî/H^{i-1}$, $G(H^\cdot\cap TV)$ be the bundle of nilpotent Lie groups $\oplus (H^i\cap TV)/(H^{i-1}\cap TV).$ In \cite{Sadegh:2016aa}, the authors define a smooth manifold whose underlying set is equal to $M\times \R^*\sqcup G(H^\cdot)_{|V}/G(H^\cdot\cap TV),$ where $G(H^\cdot)_{|V}$ is the restriction of $G(H^\cdot)$ to $V$. Similarly to the description of the classical deformation to the normal as a quotient space, the space defined in \cite{Sadegh:2016aa} can also be written as $\dnc_{H^\cdot}(M\times M,M)/\dnc_{H^\cdot\cap TV}(V\times V,V)$.
\end{enumerate}  \end{examp}
\bibliographystyle{plain}

\begin{thebibliography}{}

\end{thebibliography}


\begin{thebibliography}{10}

\bibitem{MR3261009}
Paul~F. Baum and Erik van Erp.
\newblock {$K$}-homology and index theory on contact manifolds.
\newblock {\em Acta Math.}, 213(1):1--48, 2014.

\bibitem{MR953082}
Richard Beals and Peter Greiner.
\newblock {\em Calculus on {H}eisenberg manifolds}, volume 119 of {\em Annals
  of Mathematics Studies}.
\newblock Princeton University Press, Princeton, NJ, 1988.

\bibitem{MR0219078}
N.~Bourbaki.
\newblock {\em {\'E}l{\'e}ments de math{\'e}matique. {F}asc. {XXXIII}.
  {V}ari{\'e}t{\'e}s diff{\'e}rentielles et analytiques. {F}ascicule de
  r{\'e}sultats ({P}aragraphes 1 {\`a} 7)}.
\newblock Actualit{\'e}s Scientifiques et Industrielles, No. 1333. Hermann,
  Paris, 1967.

\bibitem{MR0370271}
Louis Boutet~de Monvel.
\newblock Hypoelliptic operators with double characteristics and related
  pseudo-differential operators.
\newblock {\em Comm. Pure Appl. Math.}, 27:585--639, 1974.

\bibitem{MR0493005}
Louis Boutet~de Monvel, Alain Grigis, and Bernard Helffer.
\newblock Parametrixes d'op{\'e}rateurs pseudo-diff{\'e}rentiels {\`a}
  caract{\'e}ristiques multiples.
\newblock pages 93--121. Ast{\'e}risque, No. 34--35, 1976.

\bibitem{Choi:2015aa}
Woocheol Choi and Raphael Ponge.
\newblock Tangent maps and tangent groupoid for {C}arnot manifolds.
\newblock 10 2015.

\bibitem{Ponge:2017aa}
Woocheol Choi and Raphael Ponge.
\newblock Privileged coordinates and nilpotent approximation for {C}arnot
  manifolds, {II}. {C}arnot coordinates.
\newblock 03 2017.

\bibitem{Choi:2017aa}
Woocheol Choi and Raphael Ponge.
\newblock Privileged coordinates and nilpotent approximation of {C}arnot
  manifolds, {I}. general results.
\newblock 09 2017.

\bibitem{MR1185817}
Michael Christ, Daryl Geller, Pawe\l{} G\l{}owacki, and Larry Polin.
\newblock Pseudodifferential operators on groups with dilations.
\newblock {\em Duke Math. J.}, 68(1):31--65, 1992.

\bibitem{MR1334867}
A.~Connes and H.~Moscovici.
\newblock The local index formula in noncommutative geometry.
\newblock {\em Geom. Funct. Anal.}, 5(2):174--243, 1995.

\bibitem{MR1657389}
A.~Connes and H.~Moscovici.
\newblock Hopf algebras, cyclic cohomology and the transverse index theorem.
\newblock {\em Comm. Math. Phys.}, 198(1):199--246, 1998.

\bibitem{MR1303779}
Alain Connes.
\newblock {\em Noncommutative geometry}.
\newblock Academic Press, Inc., San Diego, CA, 1994.

\bibitem{MR973272}
Thomas~E. Cummins.
\newblock A pseudodifferential calculus associated to {$3$}-step nilpotent
  groups.
\newblock {\em Comm. Partial Differential Equations}, 14(1):129--171, 1989.

\bibitem{MR1751671}
Claire Debord.
\newblock Groupo\"\i des d'holonomie de feuilletages singuliers.
\newblock {\em C. R. Acad. Sci. Paris S{\'e}r. I Math.}, 330(5):361--364, 2000.

\bibitem{MR3187645}
Claire Debord and Georges Skandalis.
\newblock Adiabatic groupoid, crossed product by {$\Bbb{R}_+^\ast$} and
  pseudodifferential calculus.
\newblock {\em Adv. Math.}, 257:66--91, 2014.

\bibitem{Debord:2017aa}
Claire Debord and Georges Skandalis.
\newblock Blowup constructions for {L}ie groupoids and a {B}outet de {M}onvel
  type calculus.
\newblock 05 2017.

\bibitem{MR660648}
A.~Dynin.
\newblock Pseudodifferential operators on {H}eisenberg groups.
\newblock In {\em Pseudodifferential operator with applications ({B}ressanone,
  1977)}, pages 5--18. Liguori, Naples, 1978.

\bibitem{MR0423432}
A.~S. Dynin.
\newblock An algebra of pseudodifferential operators on the {H}eisenberg
  groups. {S}ymbolic calculus.
\newblock {\em Dokl. Akad. Nauk SSSR}, 227(4):792--795, 1976.

\bibitem{MR0367477}
G.~B. Folland and E.~M. Stein.
\newblock Estimates for the {$\bar \partial _{b}$} complex and analysis on the
  {H}eisenberg group.
\newblock {\em Comm. Pure Appl. Math.}, 27:429--522, 1974.

\bibitem{MR0344699}
G.~B. Folland and E.~M. Stein.
\newblock Parametrices and estimates for the {$\bar \partial _{b}$} complex on
  strongly pseudoconvex boundaries.
\newblock {\em Bull. Amer. Math. Soc.}, 80:253--258, 1974.

\bibitem{MR0442149}
Roe~W. Goodman.
\newblock {\em Nilpotent {L}ie groups: structure and applications to analysis}.
\newblock Lecture Notes in Mathematics, Vol. 562. Springer-Verlag, Berlin-New
  York, 1976.

\bibitem{MR925720}
Michel Hilsum and Georges Skandalis.
\newblock Morphismes {$K$}-orient{\'e}s d'espaces de feuilles et
  fonctorialit{\'e} en th{\'e}orie de {K}asparov (d'apr{\`e}s une conjecture
  d'{A}. {C}onnes).
\newblock {\em Ann. Sci. {\'E}cole Norm. Sup. (4)}, 20(3):325--390, 1987.

\bibitem{MR0222474}
Lars H{\"o}rmander.
\newblock Hypoelliptic second order differential equations.
\newblock {\em Acta Math.}, 119:147--171, 1967.

\bibitem{MR1332908}
Pierre Julg and Gennadi Kasparov.
\newblock Operator {$K$}-theory for the group {${\rm SU}(n,1)$}.
\newblock {\em J. Reine Angew. Math.}, 463:99--152, 1995.

\bibitem{MR3694098}
Pierre Julg and Erik van Erp.
\newblock The geometry of the osculating nilpotent group structures of the
  {H}eisenberg calculus.
\newblock {\em J. Lie Theory}, 28(1):107--138, 2018.

\bibitem{MR896907}
K.~Mackenzie.
\newblock {\em Lie groupoids and {L}ie algebroids in differential geometry},
  volume 124 of {\em London Mathematical Society Lecture Note Series}.
\newblock Cambridge University Press, Cambridge, 1987.

\bibitem{MR2247877}
Rapha{\"e}l Ponge.
\newblock The tangent groupoid of a {H}eisenberg manifold.
\newblock {\em Pacific J. Math.}, 227(1):151--175, 2006.

\bibitem{MR2417549}
Rapha{\"e}l~S. Ponge.
\newblock Heisenberg calculus and spectral theory of hypoelliptic operators on
  {H}eisenberg manifolds.
\newblock {\em Mem. Amer. Math. Soc.}, 194(906):viii+ 134, 2008.

\bibitem{MR941624}
Jean Pradines.
\newblock Remarque sur le groupo\"\i de cotangent de {W}einstein-{D}azord.
\newblock {\em C. R. Acad. Sci. Paris S{\'e}r. I Math.}, 306(13):557--560,
  1988.

\bibitem{MR0436223}
Linda~Preiss Rothschild and E.~M. Stein.
\newblock Hypoelliptic differential operators and nilpotent groups.
\newblock {\em Acta Math.}, 137(3-4):247--320, 1976.

\bibitem{MR1267892}
Michel Rumin.
\newblock Formes diff{\'e}rentielles sur les vari{\'e}t{\'e}s de contact.
\newblock {\em J. Differential Geom.}, 39(2):281--330, 1994.

\bibitem{Sadegh:2016aa}
Ahmad Reza Haj~Saeedi Sadegh and Nigel Higson.
\newblock Euler-like vector fields, deformation spaces and manifolds with
  filtered structure.
\newblock 11 2016.

\bibitem{MR764508}
Michael~E. Taylor.
\newblock Noncommutative microlocal analysis. {I}.
\newblock {\em Mem. Amer. Math. Soc.}, 52(313):iv+182, 1984.

\bibitem{MR2680395}
Erik van Erp.
\newblock The {A}tiyah-{S}inger index formula for subelliptic operators on
  contact manifolds. {P}art {I}.
\newblock {\em Ann. of Math. (2)}, 171(3):1647--1681, 2010.

\bibitem{MR2680396}
Erik van Erp.
\newblock The {A}tiyah-{S}inger index formula for subelliptic operators on
  contact manifolds. {P}art {II}.
\newblock {\em Ann. of Math. (2)}, 171(3):1683--1706, 2010.

\bibitem{MR2746652}
Erik van Erp.
\newblock The index of hypoelliptic operators on foliated manifolds.
\newblock {\em J. Noncommut. Geom.}, 5(1):107--124, 2011.

\bibitem{Erp:2015aa}
Erik van Erp and Robert Yuncken.
\newblock A groupoid approach to pseudodifferential operators.
\newblock{\em J. Reine Angew. Math. 756 (2019), 151–182}

\bibitem{Erp:2016aa}
Erik van Erp and Robert Yuncken.
\newblock On the tangent groupoid of a filtered manifold.
\newblock {\em Bull. Lond. Math. Soc.}, 49(6):1000--1012, 2017.

\bibitem{PCaro}
P. Carrillo Rouse.
\newblock A Schwartz type algebra for the tangent groupoid.
\newblock {\em In K-theory and noncommutative geometry, EMS Ser. Congr. Rep., pages 181–199. Eur. Math. Soc., Zurich, 2008.}

\end{thebibliography}

\end{document}